\documentclass[11pt]{amsart}
\usepackage{amsmath,amssymb,amsthm}
\usepackage{graphicx}
\usepackage{cite,mathrsfs}
 \usepackage[T1]{fontenc}
 \usepackage{textcomp}
 \usepackage{lmodern}

\newtheorem{theorem}{Theorem}[section]
\newtheorem{lemma}[theorem]{Lemma}

\theoremstyle{definition}

\newtheorem{assumption}[theorem]{Assumption}

\theoremstyle{remark}
\newtheorem{remark}[theorem]{Remark}

\numberwithin{equation}{section}

\begin{document}

\title[Symmetric power $L$-functions]{On the value-distribution of
symmetric power $L$-functions}

\author{Kohji Matsumoto}
\address{K. Matsumoto: Graduate School of Mathematics, Nagoya University, Chikusa-ku, Nagoya 464-8602, Japan}
\email{kohjimat@math.nagoya-u.ac.jp}

\author{Yumiko Umegaki}
\address{Y. Umegaki:
Faculty, Division of Natural Sciences, Research Group of Mathematics,
Nara Women's University, Kitauoya Nishimachi, Nara 630-8506, Japan}
\email{ichihara@cc.nara-wu.ac.jp}

\keywords{symmetric power $L$-function, automorphic $L$-function, 
value-distribution, density function, $M$-function}
\subjclass[2010]{Primary 11F66, Secondary 11M41}
%

\begin{abstract}
We first briefly survey the value-distribution theory of $L$-functions of the
Bohr-Jessen flavor (or the theory of ``$M$-functions''). 
Limit formulas for the Riemann zeta-function, Dirichlet $L$-functions, automorphic
$L$-functions etc. are discussed.
Then we prove new results on the value-distribution of symmetric power
$L$-functions, which are limit formulas involving associated $M$-functions. 
\end{abstract}

\maketitle

\section{The Bohr-Jessen limit theorem}\label{sec1}

We begin with the classical result of Bohr and Jessen \cite{BJ3032} on the
value-distribution of the Riemann zeta-function $\zeta(s)$.

Let $R$ be a rectangle in the complex plane $\mathbb{C}$ with the edges parallel
to the axes.    Let $s=\sigma+it$ be a complex variable.
By $\mu_1$ we mean the $1$-dimensional Lebesgue measure.
For $\sigma>1/2$ and $T>0$, we define
\begin{equation}\label{1-1}
V_{\sigma}(T,R;\zeta)=\mu_1\{t\in[-T,T]\;|\;\log\zeta(\sigma+it)\in R\},
\end{equation} 
where the rigorous definition of $\log\zeta(\sigma+it)$ will be given later
(in Section \ref{sec2}).
Then the result of Bohr and Jessen can be stated as follows.

\begin{theorem}[Bohr and Jessen \cite{BJ3032}]\label{thm-BJ}
%
\begin{itemize}
\item[(i)] There exists the limit
\begin{equation}\label{1-2}
W_{\sigma}(R;\zeta)=\lim_{T\to\infty}\frac{1}{T}V_{\sigma}(T,R;\zeta).
\end{equation}
\item[(ii)]
This limit can be written as
\begin{equation}\label{1-3}
W_{\sigma}(R;\zeta)=\int_R \mathcal{F}_{\sigma}(z,\zeta)|dz|,
\end{equation}
where $z=x+iy\in\mathbb{C}$, $|dz|=dxdy/2\pi$, 
and $\mathcal{F}_{\sigma}(z,\zeta)$ is a continuous non-negative, explicitly
constructed function defined on $\mathbb{C}$.
\end{itemize}
\end{theorem}

The limit $W_{\sigma}(R;\zeta)$ may be regarded as the probability of how many
values of $\log\zeta(\sigma+it)$ on the line $\Re s=\sigma$ belong to the given
rectangle $R$, and $\mathcal{F}_{\sigma}(z,\zeta)$ may be called the density
function of this probability.     Theorem \ref{thm-BJ} is now called the
Bohr-Jessen limit theorem.

\begin{remark}
A reformulation of this type of results in terms of weak convergence of probability
measures was given by Laurin{\v c}ikas (see \cite{Lau96}).
\end{remark}

The original proof of Bohr and Jessen is of some geometric flavor.
Their proof starts with the expression
\begin{equation}\label{1-4}
\log\zeta(\sigma+it)=-\sum_{n=1}^{\infty}\log(1-p_n^{-\sigma-it}),
\end{equation}
where $p_n$ is the 
$n$th prime, which is valid for $\sigma>1$.
They consider the truncation
\begin{equation}\label{1-5}
f_N(\sigma+it)=-\sum_{n=1}^N\log(1-p_n^{-\sigma-it})
    =-\sum_{n=1}^N\log(1-p_n^{-\sigma}e^{-it\log p_n}),
\end{equation}
which, even in the case $1/2<\sigma\leq 1$, approximates $\log\zeta(\sigma+it)$ 
in a certain mean value sense.    A key idea of Bohr and Jessen is to introduce
the auxiliary mapping $S_N:\mathbb{T}^N\to\mathbb{C}$ associated with
$f_N(\sigma+it)$ (where $\mathbb{T}^N\simeq [0,1)^N$ is the $N$-dimensional unit
torus) defined by
\begin{equation}\label{1-6}
S_N(\theta_1,\ldots,\theta_N;\zeta)=-\sum_{n=1}^N\log(1-p_n^{-\sigma}e^{2\pi i\theta_n})
\qquad (0\leq \theta_n <1).
\end{equation}
Let $z_n(\theta;\zeta)=-\log(1-p_n^{-\sigma}e^{2\pi i\theta})$.
Then each term $z_n(\theta_n;\zeta)$ on the right-hand side of \eqref{1-6}
describes a planer convex curve when $\theta_n$ varies from 0 to 1.
Therefore $S_N(\theta_1,\ldots,\theta_N;\zeta)$ is a kind of geometric ``sum'' of convex
curves.
Bohr and Jessen \cite{BJ29} developed a detailed theory on such sums of convex
curves, and applied it to the proof of their Theorem \ref{thm-BJ}.

Later Jessen and Wintner \cite{JW35} published an alternative proof of
Theorem \ref{thm-BJ}, which is more analytic (Fourier theoretic).
In their proof they used a certain inequality
(the Jessen-Wintner inequality), which is also related with convex 
properties of curves.

We also note that the analogue of Theorem \ref{thm-BJ} for $(\zeta'/\zeta)(s)$
was shown by Kershner and Wintner \cite{KW37}.     As for the explicit construction 
of the density function, see van Kampen and Wintner \cite{vKW37}.

\section{A generalization of the Bohr-Jessen limit theorem}\label{sec2}

It is a natural question to ask how to generalize Theorem \ref{thm-BJ}, the Bohr-Jessen 
limit theorem, to more
general zeta and $L$-functions.    An obstacle is that, in more general situation,
the geometry of corresponding curves becomes more complicated; especially, the convexity is
not valid in general.

Still, however, the part (i) of 
Theorem~\ref{thm-BJ} can be generalized to a fairly
general class of zeta-functions. 

Let $\mathbb{N}$ be the set of positive integers.     For any $n\in\mathbb{N}$, 
let $g(n)\in \mathbb{N}$, $f(k,n)\in\mathbb{N}$ and $a_n^{(k)}\in\mathbb{C}$
($1\leq k\leq n$).
Using the polynomilas given by
\[
A_n(X)=\prod_{k=1}^{g(n)}\left(1-a_n^{(k)}X^{f(k,n)}\right),
\]
we define the zeta-function $\varphi(s)$ by the Euler product
\begin{equation}\label{2-1}
\varphi(s)=\prod_{n=1}^{\infty} A_n(p_n^{-s})^{-1}.
\end{equation}
Assume 
\begin{equation}\label{2-2}
g(n)\leq C_0 p_n^{\alpha}, \qquad |a_n^{(k)}|\leq p_n^{\beta}
\end{equation}
with constants $\alpha,\beta\geq 0$, $C_0 >0$.
Then \eqref{2-1} is convergent absolutely in the region $\Re s>\alpha+\beta+1$.

Let
$\mathscr{M}_{\alpha\beta}$
be the set of all functions $\varphi(s)$ defined as above, satisfying \eqref{2-2}
and the following:
\begin{itemize}
\item[(i)] $\varphi(s)$ can be continued meromorphically to $\sigma\geq\sigma_0$,
where $\alpha+\beta+1/2\leq\sigma_0<\alpha+\beta+1$,
and all poles in this region are
included in a compact subset of $\{s\;|\;\sigma>\sigma_0\}$,
\item[(ii)] 
  $\varphi(\sigma+it)=O((|t|+1)^{C'_0})$
  for any $\sigma\geq\sigma_0$, with a constant 
  $C'_0>0$,
\item[(iii)] It holds that
\begin{equation}\label{2-3}
\int_{-T}^T |\varphi(\sigma_0+it)|^2 dt = O(T).
\end{equation}
\end{itemize}

The class 
\[
\mathscr{M}=\bigcup_{\alpha,\beta\geq 0}\mathscr{M}_{\alpha\beta}
\]
was first introdued by the first author \cite{Mat90}.
For $\sigma >\sigma_0$, 
define
\begin{equation}\label{2-4}
V_{\sigma}(T,R;\varphi)=\mu_1\{t\in[-T,T]\;|\;\log\varphi(\sigma+it)\in R\},
\end{equation} 
where the definition of $\log\varphi(s)$ (for $\varphi\in\mathscr{M}$) is as follows.    First, when
$\sigma>\alpha+\beta+1$ define
\[
\log\varphi(s)=-\sum_{n=1}^{\infty}\sum_{k=1}^{g(n)}\mathrm{Log}\left(1-a_n^{(k)}
p_n^{-f(k,n)s}\right),
\]
where $\mathrm{Log}$ means the principal branch.
Next, let
\[
G(\varphi)=\{s\;|\;\sigma\geq\sigma_0\}\setminus\bigcup_{\rho}
\{\sigma+i\Im\rho\;|\;\sigma_0\leq\sigma\leq\Re\rho\},
\]
where $\rho$ runs over all zeros and poles $\rho$ with $\Re\rho\geq\sigma_0$.
For any $s\in G(\varphi)$, define $\log\varphi(s)$ by the analytic continuation along
the horizontal path from the right.

In this general situation, the corresponding mapping is
\begin{equation}\label{2-5}
S_N(\theta_1,\ldots,\theta_N;\varphi)=\sum_{n=1}^N z_n(\theta_n;\varphi)
\qquad 
(0\leq \theta_n <1),
\end{equation}
where
\begin{equation}\label{2-6}
z_n(\theta_n;\varphi)=-\sum_{k=1}^{g(n)}
\log(1-a_n^{(k)}p_n^{-f(k,n)\sigma}e^{2\pi if(k,n)\theta_n}).
\end{equation}

In \cite{Mat90}, the following generalization of Theorem \ref{thm-BJ} (i) was shown.

\begin{theorem}[\!\!\cite{Mat90}]\label{thm-M90}
If $\varphi\in\mathscr{M}$,
then for any $\sigma>\sigma_0$, the limit
\begin{equation}\label{2-7}
W_{\sigma}(R;\varphi)=\lim_{T\to\infty}\frac{1}{2T}V_{\sigma}(T,R;\varphi)
\end{equation}
exists.
\end{theorem}

It can be seen that the class 
$\mathscr{M}$ includes a lot of important zeta and
$L$-functions.     The reason why such general statement can be shown is that, 
for the proof of this theorem, geometric properties of corresponding curves
\eqref{2-6} are not necessary.
In fact, the proof of Theorem \ref{thm-M90} is just based on (besides simple
arithmetic facts) Prokhorov's theorem in probability theory. 

An alternative proof is given in \cite{Mat92a} in the case of Dedekind zeta-functions of
algebraic number fields.     The method in \cite{Mat92a} is to use L{\'e}vy's
convergence theorem, again in probability theory.     This method can also be applied
to general 
$\varphi\in\mathscr{M}$,
which is pointed out in \cite{Mat92b} and a sketch
of the argument in the general case is described in \cite{MU2}.

Therefore, now we can say that the part (i) of Theorem \ref{thm-BJ} has been
sufficiently generalized.     However Theorem \ref{thm-BJ} includes the part (ii).
The part (ii) gives an explicit expression of the limit value in terms of the density
function, so it is highly
desirable to generalize also the part (ii), in order to study the behavior of the
limit $W_{\sigma}(R;\varphi)$ more closely.

However this part is related with the geometry of corresponding curves, and its
generalization is much more difficult.
Joyner \cite{Joy86} discussed the properties of density functions in the case 
of Dirichlet $L$-functions, and the first author \cite{Mat92a} studied the density
functions for
Dedekind zeta-functions of Galois number fields, but both of them are the cases when
the corresponding curves \eqref{2-6} are convex. 

In the case of automorphic $L$-functions, the corresponding \eqref{2-6} is not
always convex.     The study in this case will be given in later sections.

\section{$M$-functions}\label{sec3}

The theorems of Bohr-Jessen type consider the situation when $t=\Im s$ varies.    That is, 
Theorems \ref{thm-BJ} and \ref{thm-M90} are results in $t$-aspect.
When we consider more general zeta and $L$-functions, it is also important to study
the value-distribution in some different aspect.     
For example, it is possible to consider the modulus aspect for Dirichlet 
or Hecke $L$-functions.

Let $\chi$ be a certain character, and $L(s,\chi)$ be the 
associated $L$-function (over certain number field or function field).    
Ihara \cite{Iha08} studied the behavior of
$(L'/L)(s,\chi)$ from this aspect, and proved the limit formula of the form
\begin{equation}\label{3-1}
\mathrm{Avg}_{\chi}\Phi\left(\frac{L'}{L}(s,\chi)\right)=\int_{\mathbb{C}}M_{\sigma}(z)
\Phi(z)|dz|
\end{equation}
for a certain average (specified below) with respect to $\chi$, where $\Phi$ is a
test function, and $M_{\sigma}:\mathbb{C}\to\mathbb{R}$ is an explicitly constructed 
density function, which is non-negative, and belongs to the class $C^{\infty}$.
Ihara called this $M_{\sigma}$ the ``$M$-function'' associated with the
value-distribution of $L(s,\chi)$.

When $\sigma>1$, Ihara proved \eqref{3-1} for any continuous test function $\Phi$.
In the function field case, using the (proved) Riemann hypothesis, Ihara proved
\eqref{3-1} even in some subregion in the critical strip for more restricted class of
$\Phi$ (e.g. $\sigma>3/4$ when $\Phi\in L^1\cap L^{\infty}$ and moreover its Fourier
transform has compact support).

As for the meaning of $\mathrm{Avg}_{\chi}$, Ihara considered several types of averages,
but when the ground field is the rational number field $\mathbb{Q}$, the meaning is
one of the following: The first type is
\begin{equation}\label{3-2}
\mathrm{Avg}_{\chi}\phi(\chi)=\lim_{m\to\infty}\frac{1}{\pi(m)}
\sum_{2< p\leq m}\frac{1}{p-2}
{\sum_{\chi (\text{mod}\; p)}}^{\!\!\!\!\!*} \phi(\chi)
\end{equation}
for a complex-valued function $\phi$ of $\chi$, 
where $\pi(m)$ denotes the number of primes up to $m$, $p$ runs over primes, and $\sum^*$ 
stands for the sum over primitive Dirichlet characters of modulus $p$.
The second type is considered for the character $\chi_{\tau}(p)=p^{-i\tau}$.     
Then the Euler product
of the associated $L$-function is
\[
\prod_p(1-\chi_{\tau}(p)p^{-s})^{-1}
=\prod_p(1-p^{-s-i\tau})^{-1}=\zeta(s+i\tau),
\]
and the meaning of the average is given by
\begin{equation}\label{3-3}
\mathrm{Avg}_{\chi}\phi(\chi)=\lim_{T\to\infty}\frac{1}{2T}\int_{-T}^T \phi(\chi_{\tau})
d\tau.
\end{equation}

The second type of average actually implies a limit formula for the Riemann zeta-function
in $t$-aspect.    In particular, the formula \eqref{3-1} for this second type of average,
with $\Phi$ being the characteristic function of $R$, coincides with the formulation of
Kershner and Wintner \cite{KW37}.    An important discovery in Ihara \cite{Iha08} is
that the same function $M_{\sigma}$ can be used in the formula \eqref{3-1} for both of
the meanings of average.

Now we restrict ourselves to the case when the ground field is $\mathbb{Q}$, so the
meaning of the average is \eqref{3-2} or \eqref{3-3}.     We also consider the 
value-distribution of $\log L(s,\chi)$, so the corresponding limit formula is of
the form
\begin{equation}\label{3-4}
\mathrm{Avg}_{\chi}\Phi\left(\log L(s,\chi)\right)=\int_{\mathbb{C}}\mathcal{M}_{\sigma}(z)
\Phi(z)|dz|
\end{equation}
with the density function $\mathcal{M}_{\sigma}$.

\begin{theorem}[Ihara and Matsumoto \cite{IM11a} \cite{IM14}]\label{thm-IM1114}
For any $\sigma> 1/2$, and for the average \eqref{3-2} or \eqref{3-3}, both
\eqref{3-1} and \eqref{3-4} hold with explicitly constructed density functions
(``$M$-functions'')
$M_{\sigma}$ and $\mathcal{M}_{\sigma}$, for any test function $\Phi$ which is
(i) any bounded continuous function, or (ii) the characteristic function of either a
compact subset of $\mathbb{C}$ or the complement of such a subset.
\end{theorem}

In the number field case the Riemann hypothesis is surely not yet proved, but instead,
we can apply certain mean value estimates to obtain the above theorem.
Therefore Theorem \ref{thm-IM1114} is unconditional.    In particular, this theorem
includes the Bohr-Jessen limit theorem, and its $\zeta'/\zeta$ analogue due to
Kershner and Wintner, as special cases.

If we assume the Riemann hypothesis (for Dirichlet $L$-functions), even stronger result
can be shown.     In \cite{IM11b}, the average
\begin{equation}\label{3-5}
\mathrm{Avg}_{\chi}\phi(\chi)=\lim_{p\to\infty}\frac{1}{p-2}
{\sum_{\chi (\text{mod}\; p)}}^{\!\!\!\!\!*} \phi(\chi)
\end{equation}
was considered, and for this average, both \eqref{3-1} and \eqref{3-4} were proved for
more general class of test functions (that is, (i) of Theorem~\ref{3-1} is replaced by any continuous
function with at most exponential growth) under the assumption of the Riemann hypothesis.

The corresponding study for $M$-functions in the function field case was done in 
\cite{IM10} \cite{IM11b}.

Here we mention several further researches in the theory of $M$-functions.
Let $D$ a fundamental discriminant, and $\chi_D$ the associated real character.
Mourtada and Murty \cite{MM15} studied the value-distribution of $(L'/L)(\sigma,\chi_D)$
(where $\sigma>1/2$) in $D$-aspect, and proved a limit formula similar to \eqref{3-1}
under the assumption of the Riemann hypothesis.
Akbary and Hamieh \cite{AH18} proved an analogous result for the cubic character case,
without the assumption of the Riemann hypothesis.

As for the value-distribution of $(\zeta'/\zeta)(s)$ in $t$-aspect, there is another
approach due to Guo \cite{Guo96a} \cite{Guo96b}.     Inspired by the idea of Guo,
Mine \cite{Mine1} proved the existence (and the explicit construction) of 
the $M$-function for $(\zeta'_K/\zeta_K)(s)$
in $t$-aspect, where $\zeta_K(s)$ denotes the Dedekind zeta-function of an algebraic
number field $K$ (including the non-Galois case), with an explicit error estimate in
the limit formula of the form \eqref{3-1}.
In \cite{Mine3}, he extended the result to the case of more general $L$-functions,
belonging to a certain subclass of 
$\mathscr{M}$.

In his another paper \cite{Mine2}, Mine treated the limit theorem of Bohr-Jessen type
(but without taking logarithm) for Lerch zeta-functions, and proved a refinement,
written in terms of the associated $M$-function.    This paper of Mine implies that
the theory of $M$-functions works for zeta-functions without Euler products.

Suzuki \cite{Suz15} discovered that certain $M$-function appears even in a rather
different context.     He studied the zeros of the real or imaginary part of
\[
\xi(s)=\frac{1}{2}s(s-1)\pi^{-s/2}\Gamma(s/2)\zeta(s),
\]
and proved that the distribution of spacings of the second-order normalization of 
imaginary parts of those zeros can be represented by an integral involving the
$M$-function for $(\zeta'/\zeta)(s)$.

\section{The value-distribution of automorphic $L$-functions (the modulus and level
aspects)}\label{sec4}

At the end of the preceding section we watched that $M$-functions have been studied
for various zeta and $L$-functions.     Since one of the most important classes of
$L$-functions is the class of automorphic $L$-functions, it is natural to ask how is
the theory of $M$-functions associated with automorphic $L$-functions.

First we fix the notation.
Let $k$ be an even integer and $N$ positive integer,
and let $S_k(N)$ be the set of cusp forms of weight $k$ 
for $\Gamma_0(N)$.
We write the Fourier expansion of $f\in S_k(N)$ as
\[
f(z)=\sum_{n=1}^{\infty}\lambda_f(n)n^{(k-1)/2}e^{2\pi inz},
\]
and define the attached $L$-function by
\[
L(f,s)=\sum_{n=1}^{\infty}\lambda_f(n)n^{-s}.
\]
Now we assume that $f\in S_k(N)$ is a primitive form, that is, a normalized Hecke-eigen
newform.    Then $L(s,f)$ has the Euler product
\begin{align}\label{4-1}
L(f,s)&=\prod_{p|N}(1-\lambda_f(p)p^{-s})^{-1}
   \prod_{p\nmid N}(1-\lambda_f(p)p^{-s}+p^{-2s})^{-1}\\
&=\prod_{p|N}(1-\lambda_f(p)p^{-s})^{-1}
   \prod_{p\nmid N}(1-\alpha_f(p)p^{-s})^{-1}(1-\beta_f(p)p^{-s})^{-1},\notag
\end{align}
where $|\alpha_f(p)|=1$, $\beta_f(p)=\overline{\alpha_f(p)}$, and
$\alpha_f(p)+\beta_f(p)=\lambda_f(p)$ (for $p\nmid N$).

First consider the modulus aspect.
Let $\chi$ be a Dirichlet character.
The twisted $L$-function $L(f\otimes\chi,s)$ is defined by replacing
$p^{-s}$ by $\chi(p)p^{-s}$ on each local factor.
Lebacque and Zykin \cite{LZ} developed the theory similar to \cite{IM11b} for
$L(f\otimes\chi,s)$, and proved the limit formulas corresponding to \eqref{3-1}
and \eqref{3-4}.

More difficult is the case of the level aspect.    So far there are two attempts
in this direction, the aforementioned paper of Lebacque and Zykin \cite{LZ}, and an
article of the authors \cite{MU1}.    Here we briefly mention the results proved in
\cite{MU1}.

Let $\gamma\in\mathbb{N}$, and define the (partial) $\gamma$th symmetric power $L$-function
attached to $f$ by
\begin{equation}\label{4-2}
L_N(\mathrm{Sym}_f^{\gamma},s)=
\prod_{p\nmid N}\prod_{h=0}^{\gamma}(1-\alpha_f^{\gamma-h}(p)\beta_f^h(p)p^{-s})^{-1}.
\end{equation}

Here we consider the situation $N=q^m$, where $q$ is a prime number.
Then the form of the right-hand side of \eqref{4-2} is the same for all $m$, which we 
denote by
\begin{equation}\label{4-3}
L_q(\mathrm{Sym}_f^{\gamma},s)=
\prod_{p\neq q}\prod_{h=0}^{\gamma}(1-\alpha_f^{\gamma-h}(p)\beta_f^h(p)p^{-s})^{-1}.
\end{equation}
Let $\mu,\nu\in\mathbb{N}$, $\mu-\nu=2$.    
By $Q(\mu)$ we denote the smallest prime number satisfying $2^{\mu}/\sqrt{Q(\mu)}<1$.
The main results in \cite{MU1} is the limit 
formula for the value-distribution of the difference
\[
\log L_q(\mathrm{Sym}_f^{\mu},\sigma)-\log L_q(\mathrm{Sym}_f^{\nu},\sigma)
\qquad(\sigma>1/2).
\]

In the proof of limit theorems mentioned in the present article, some kind of
``independence'' or ``orthogonality'' properties are necessary.
For example, in the proof of Theorem \ref{thm-BJ} and Theorem \ref{thm-M90}, the
Kronecker-Weyl theorem on the uniform distribution of sequences is used.
Ihara's argument \cite{Iha08} for $L$-functions is based on the orthogonality of
Dirichlet characters.
In the present situation, the necessary tool is supplied by Petersson's formula,
in the form shown in the second author's article \cite{Ich11}.
In view of the formula in \cite{Ich11}, we define the following weighted sum
for any sequence $\{A_f\}$ over primitive forms $f\in S_k(q^m)$:
\begin{equation}\label{4-4}
{\sum_f}^{\prime}A_f= \frac{1}{C_k(1-C_q(m))}\sum_f \frac{A_f}{\langle f,f\rangle_P},
\end{equation}
where 
\[
C_k=\frac{(4\pi)^{k-1}}{\Gamma(k-1)},\quad
C_q(m)=
\begin{cases}
    0, & m=1,\\
    q(q^2-1)^{-1},  & m=2,\\
    q^{-1},  &  m\geq 3,
\end{cases}
\]
the symbol $\langle,\rangle_P$ is the Petersson inner product, and the sum on the right-hand 
side of \eqref{4-4} runs over all primitive forms belonging to $S_k(q^m)$.

We define two types of averages in the level aspect.    The first one is
\begin{align}\label{4-5}
\lefteqn{\mathrm{Avg}_{\text{prime}}
\Psi(\log L_q(\mathrm{Sym}_f^{\mu},\sigma)-\log L_q(\mathrm{Sym}_f^{\nu},\sigma))}\\
&=\lim_{q\to\infty}{\sum_f}^{\prime}
\Psi(\log L_q(\mathrm{Sym}_f^{\mu},\sigma)-\log L_q(\mathrm{Sym}_f^{\nu},\sigma))\notag
\end{align}
for fixed $m$, where $\Psi:\mathbb{R}\to\mathbb{C}$ is a test function.
The second one is
\begin{align}\label{4-6}
\lefteqn{\mathrm{Avg}_{\text{power}}
\Psi(\log L_q(\mathrm{Sym}_f^{\mu},\sigma)-\log L_q(\mathrm{Sym}_f^{\nu},\sigma))}\\
&=\lim_{m\to\infty}{\sum_f}^{\prime}
\Psi(\log L_q(\mathrm{Sym}_f^{\mu},\sigma)-\log L_q(\mathrm{Sym}_f^{\nu},\sigma))\notag
\end{align}
for fixed $q$, where $q$ is a prime and $q\geq Q(\mu)$ if $1\geq\sigma >1/2$.

\begin{theorem}[\!\!\cite{MU1}]\label{thm-MU1}
Let $f\in S_k(N)$ be a primitive form, $2\leq k<12$ or $k=14$, and $N=q^m$ with a certain
prime $q$.   Let $\mu,\nu\in\mathbb{N}$, $\mu-\nu=2$.    We assume that the symmetric
power L-functions 
$L_q(\mathrm{Sym}_f^{\mu},s)$,
$L_q(\mathrm{Sym}_f^{\nu},s)$ can be continued
holomorphically to $\sigma>1/2$, satisfy the estimate
$\ll q^m(|t|+2)$ in the strip $2\geq\sigma>1/2$, and have no zero in
$1\geq\sigma>1/2$.    Then, for any $\sigma>1/2$, there exists a density function
$\mathcal{M}_{\sigma}:\mathbb{R}\to\mathbb{R}_{\geq 0}$ which can be explicitly
constructed, and for which the formula
\begin{align}\label{4-7}
\mathrm{Avg}_{\mathrm{prime}}
\lefteqn{\Psi(\log L_q(\mathrm{Sym}_f^{\mu},\sigma)-\log L_q(\mathrm{Sym}_f^{\nu},\sigma))}\\
&=\mathrm{Avg}_{\mathrm{power}}
\Psi(\log L_q(\mathrm{Sym}_f^{\mu},\sigma)-\log L_q(\mathrm{Sym}_f^{\nu},\sigma))\notag\\
&=\int_{\mathbb{R}}\mathcal{M}_{\sigma}(u)\Psi(u)\frac{du}{\sqrt{2\pi}}\notag
\end{align}
holds for any test function $\Psi$ which is bounded continuous, or a compactly
supported characteristic function.
\end{theorem}

In this theorem we require several assumptions, which are plausible but seem very
difficult to prove.     The main reason of using those assumptions is that we have
no idea of showing suitable mean value estimates for symmetric power $L$-functions.
\par
For any $\sigma >1$, since $\mu-\nu=2$, we have
\begin{align*}
&  \log L_q(\textrm{Sym}_f^{\mu}, s) - \log L_q(\textrm{Sym}_f^{\nu}, s)
\\
=&\sum_{p\neq q}(-\log(1-\alpha_f^{\mu}(p)p^{-s})
  -\log(1-\beta_f^{\mu}(p)p^{-s})).
\end{align*}
%
%
If we could find a method for the study of $\mathrm{Avg}_{\text{prime}}$ and
$\mathrm{Avg}_{\text{power}}$ of the right-hand side of the above equation
in the case $\mu=1$,
it would imply the limit theorem for
$\log L(f,s)$ similar to \eqref{3-4},
but at present we cannot extend the theorem to 
$\log L(f,s)$.
(The theorem is shown only for $\mu\geq 3$.)

Lebacque and Zykin \cite{LZ} studied $\log L(f,s)$ and $(L'/L)(f,s)$ along the line of
\cite{IM11b}, and obtained a result analogous to \cite[Theorem 1]{IM11b}.
However their argument also does not arrive at the limit theorem for
$\log L(f,s)$ or $(L'/L)(f,s)$ of the form \eqref{3-1} or \eqref{3-4}.

\section{The value-distribution of automorphic $L$-functions (the 
$t$-aspect)}\label{sec5}

Now we return to the matter of $t$-aspect.     As we mentioned in Section \ref{sec2},
the part (ii) of Theorem \ref{thm-BJ} has been generalized 
only for some special cases when convex properties can be used.

Automorphic $L$-functions are typical examples for which the corresponding curves are
not always convex, so it is important how to generalize the part (ii) of
Theorem \ref{thm-BJ} to the case of automorphic $L$-functions $L(f,s)$.
This has been done in \cite{MU2}.

Since 
$L(f,s)\in\mathscr{M}_{00}$,
the existence of the limit
$W_{\sigma}(R;L(f,\cdot))$ (for $\sigma>1/2$) is already known by Theorem \ref{thm-M90}.

\begin{theorem}[\!\!\cite{MU2}]\label{thm-MU2}
For any $\sigma>1/2$, there exists a continuous non-negative function
${\mathcal M}_{\sigma}(z,L(f,\cdot))$, explicitly defined on $\mathbb{C}$, for which
\begin{equation}\label{5-1}
W_{\sigma}(R;L(f,\cdot))=\int_R {\mathcal M}_{\sigma}(z,L(f,\cdot))|dz|
\end{equation}
holds.
\end{theorem}

\begin{remark}\label{rem-MU2}
Once \eqref{5-1} is proved, then we can deduce
\begin{align}\label{5-1bis}
\lim_{T\to\infty}\frac{1}{2T}\int_{-T}^T \Phi(\log L(f,s+i\tau))d\tau
=\int_{\mathbb{C}}{\mathcal M}_{\sigma}(z,L(f,\cdot))\Phi(z)|dz|
\end{align}
for any test function $\Phi$ as in the statement of Theorem \ref{thm-IM1114}, by the 
argument given in \cite[Remark 9.1]{IM11a}.
Note that the left-hand side of \eqref{5-1bis} is the function of variable $s$,
but the $M$-function on the right-hand side depends only on $\sigma=\Re s$.
\end{remark}

The basic structure of the proof of Theorem \ref{thm-MU2} in \cite{MU2},
which we briefly outline here, is 
along the line similar to \cite{Mat92a}.
Actually in \cite{MU2}, we are working in more general situation, that is in the class 
$\mathscr{M}$.
Let $\varphi\in\mathscr{M}$.
Define the integral
\begin{equation}\label{5-2}
K_n(w;\varphi)=\int_0^1 \exp\left(i\langle z_n(\theta_n;\varphi),w\rangle\right)d\theta_n
\qquad (n\in\mathbb{N}),
\end{equation}
where $z_n(\theta_n;\varphi)$ is defined by \eqref{2-6} and
$\langle z,w\rangle=\Re z \Re w + \Im z \Im w$.
In \cite{MU2}, we prove the following

\begin{lemma}[\!\!\cite{MU2}]\label{lem-MU2}
If there are at least five $n$'s for which
\begin{equation}\label{5-3}
K_n(w;\varphi)=O_n(|w|^{-1/2}) \qquad (|w|\to\infty)
\end{equation}
holds, then we can find a continuous non-negative function
${\mathcal M}_{\sigma}(z,\varphi)$ for $\sigma>\sigma_0$
by which we can write
\begin{equation}\label{5-4}
W_{\sigma}(R;\varphi)=\int_R {\mathcal M}_{\sigma}(z,\varphi) |dz|.
\end{equation}
Moreover ${\mathcal M}_{\sigma}(z,\varphi)$ is explicitly given by
\begin{equation}\label{5-5}
{\mathcal M}_{\sigma}(z,\varphi)=\int_{\mathbb{C}}e^{-i\langle z,w\rangle}
\Lambda(w;\varphi)|dw|,
\end{equation}
where
\begin{equation}\label{5-6}
\Lambda(w;\varphi)=\int_{\mathbb{C}}e^{i\langle z,w\rangle}dW_{\sigma}(z;\varphi).
\end{equation}
\end{lemma}

Therefore the main problem is reduced to the proof of \eqref{5-3}.
Jessen and Wintner \cite{JW35} proved that $K_n(w,\varphi)=O(|w|^{-1/2})$ for any $n$, when
the corresponding curves are convex.
This is the original Jessen-Wintner inequality.

Now consider the case of automorphic $L$-functions.    Let
$\mathbb{P}_f(\varepsilon)$ be the set of primes $p$ satisfying
$|\lambda_f(p)|>\sqrt{2}-\varepsilon$.    Then $\mathbb{P}_f(\varepsilon)$ is of
positive density (M. R. Murty \cite{RM83} for the full modular case, and
M. R. Murty and V. K. Murty \cite{RMKM} for any level $N$).
In \cite{MU2}, we observed geometric behavior of the corresponding curves and proved

\begin{lemma}[\!\!\cite{MU2}]\label{lem2-MU2}
If $p_n\in\mathbb{P}_f(\varepsilon)$ and $n$ is sufficiently large, then
\begin{equation}\label{5-9}
K_n(w;L(f,\cdot))=O_{\varepsilon}\left(p_n^{\sigma/2}|w|^{-1/2}
+p_n^{\sigma}|w|^{-1}\right)
\end{equation}
holds.
\end{lemma}

Since $\mathbb{P}_f(\varepsilon)$ is of positive density, obviously we can find more than
five (actually infinitely many) $n$'s for which \eqref{5-9} is valid.
Therefore using Lemma \ref{lem-MU2} we can deduce the conclusion of Theorem \ref{thm-MU2}.

\section{The value-distribution of symmetric power $L$-functions (the $t$-aspect)}
\label{sec6}

Now we proceed to state our new results in the present paper, on the value-distribution
of symmetric power $L$-functions, defined by \eqref{4-2}.
The proof of the results stated in this section will be given in Sections
\ref{sec7} and \ref{sec8}.

First consider the case $\gamma=2$, that is the symmetric square $L$-functions
\[
L(\mathrm{Sym}_f^2, s)=L_N(\mathrm{Sym}_f^2, s)\prod_{p\mid N}(1-\lambda_f(p^2)p^{-s})^{-1}.
\]
Assume $N$ is square-free and let $f\in S_k(N)$ be a primitive form.
Let
\[
\Lambda(\mathrm{Sym}_f^2,s)=N^s \pi^{-3s/2}\Gamma\left(\frac{s+1}{2}\right)
\Gamma\left(\frac{s+k-1}{2}\right)\Gamma\left(\frac{s+k}{2}\right)
L(\mathrm{Sym}_f^2,s).
\]
Then it is known (Shimura \cite{Shi75}, Gelbart and Jacquet \cite{GJ78}; see also
\cite{IM})
that $\Lambda(\mathrm{Sym}_f^2,s)$ can be continued to an entire function,
and satisfies the functional equation
\begin{align}\label{6-1}
\Lambda(\mathrm{Sym}_f^2,s)= \Lambda(\mathrm{Sym}_f^2,1-s).
\end{align}
Because of \eqref{6-1}, we can apply the general theorem of Kanemitsu, 
Sankaranarayanan and Tanigawa \cite{KST02}.    The part (iii) of their Theorem 1 implies
that, in the strip $2/3<\sigma<1$, it holds that
\begin{align}\label{6-2}
\int_1^T |L(\mathrm{Sym}_f^2,\sigma+it)|^2 dt=C_2(\sigma,f)T+
O\left(T^{2-(3/2)\sigma+\varepsilon}\right)
\end{align}
for any $\varepsilon>0$, where $C_2(\sigma,f)$ is a constant depending on $\sigma$ and $f$.
(Note that the first author \cite{Mat05} developed a more refined general theory, which
improves the error estimate in \eqref{6-2} to
$O(T^{3-3\sigma+\varepsilon})$; see \cite[(3.10)]{Mat05}.)
From \eqref{6-2} we find that
\begin{align}\label{6-3}
\int_1^T |L(\mathrm{Sym}_f^2,\sigma+it)|^2 dt=O(T)
\end{align}
for $\sigma>2/3$.    
This is condition (iii) of the class 
$\mathscr{M}_{00}$.    
Condition (ii) also follows from \eqref{6-1} by invoking the Phragm{\'e}n-Lindel{\"o}f
principle.
Therefore $L(\mathrm{Sym}_f^2,\cdot)\in\mathscr{M}_{00}$,
so the method in \cite{MU2} can be applied to
$L(\mathrm{Sym}_f^2,\cdot)$.    The result is

\begin{theorem}\label{thm-square}
Let $N$ be a square-free integer, and $f\in S_k(N)$ a primitive form.
For any $\sigma>2/3$, there exists a continuous non-negative function
${\mathcal M}_{\sigma}(z,L(\mathrm{Sym}_f^2,\cdot))$, explicitly defined on $\mathbb{C}$, for 
which
\begin{equation}\label{6-4}
\lim_{T\to\infty}\frac{1}{2T}\int_{-T}^T \Phi(\log L(\mathrm{Sym}_f^2,s+i\tau))d\tau=
\int_{\mathbb{C}} {\mathcal M}_{\sigma}(z,L(\mathrm{Sym}_f^2,\cdot))\Phi(z)|dz|
\end{equation}
holds for any test function $\Phi$ as in the statement of Theorem \ref{thm-IM1114}.
\end{theorem}

Next consider more general symmetric power $L$-functions 
$L_N(\mathrm{Sym}_f^{\gamma}, s)$ (see \eqref{4-2})
associated with a primitive form $f\in S_k(N)$, where $N$ is a positive integer.
It is known that $L_N(\mathrm{Sym}_f^{\gamma}, s)$
has meromorphic continuation to the whole complex plane (see \cite{BGHT}).
We assume the following 
\begin{assumption}\label{automorphy}
There are predicted local factors
$L_p(\mathrm{Sym}_f^{\gamma}, s)$ for $p\mid N$
and
$L_N(\mathrm{Sym}_f^{\gamma}, s)$ satisfies the functional equation
%
%
\begin{align}\label{6-5}
\Lambda(\mathrm{Sym}_f^{\gamma},s)=\varepsilon_{\gamma,f}
\Lambda(\mathrm{Sym}_f^{\gamma},1-s),
\end{align}
where $|\varepsilon_{\gamma,f}|=1$ and
\[
\Lambda(\mathrm{Sym}_f^{\gamma},s)
=q_{\gamma,f}^{s/2}\widetilde{\Gamma}_{\gamma}(s)
L_N(\mathrm{Sym}_f^{\gamma},s)\prod_{p\mid N}L_p(\mathrm{Sym}_f^{\gamma},s)
\]
with the conductor $q_{\gamma,f}$ and the ``gamma factor'' $\widetilde{\Gamma}_{\gamma}(s)$.
Here, the gamma factor is written by
\begin{align}\label{gamma-factor}
\widetilde{\Gamma}_{\gamma}(s)
=\pi^{-(\gamma+1)s/2}\prod_{j=1}^{\gamma+1}\Gamma\bigg(\frac{s+\kappa_{j,\gamma}}{2}\bigg),
\end{align}
where $\kappa_{j,\gamma}\in\mathbb{R}$, and each local factor for $p|N$ is written as
\begin{align}\label{local-factor}
L_p(\mathrm{Sym}_f^{\gamma},s)=(1-\lambda_{p,\gamma,f}p^{-s})^{-1},
\quad |\lambda_{p,\gamma,f}|\leq p^{-\gamma/2}
\end{align}
(see Cogdell and Michel \cite{CM04}, Moreno and Shahidi \cite{MS85}, Rouse \cite{rouse07},
and Rouse and Thorner \cite{RT17}).
\end{assumption}
The above assumptions are reasonable in view of the Langlands functoriality conjecture.
Let 
\[
L(\mathrm{Sym}_f^{\gamma}, s)=
L_N(\mathrm{Sym}_f^{\gamma},s)\prod_{p\mid N}L_p(\mathrm{Sym}_f^{\gamma},s).
\]
From \eqref{4-2} and \eqref{local-factor} we see that the Dirichlet series expansion
of $L(\mathrm{Sym}_f^{\gamma}, s)$ is of the form
$\sum_{n=1}^{\infty}c_n n^{-s}$, $|c_n|\ll n^{\varepsilon}$.
Since the gamma factor is given by \eqref{gamma-factor}, 
again using the general result of \cite{KST02}, we obtain
\begin{align}\label{6-7}
\int_1^T |L(\mathrm{Sym}_f^{\gamma},\sigma+it)|^2 dt=C_{\gamma}(\sigma,f)T+
O\left(T^{1+(\gamma/2)-((\gamma+1)/2)\sigma+\varepsilon}\right)
\end{align}
in the strip 
$1-1/(\gamma+1) <\sigma<1$,
with a certain constant 
$C_{\gamma}(\sigma,f)$.
Therefore $L(\mathrm{Sym}_f^{\gamma},\cdot)\in\mathscr{M}_{00}$.

Another tool we use is
the following quantitative version of the Sato-Tate conjecture
due to Thorner \cite{Tho14}.
We write $\alpha_f(p)=e^{i\theta_f(p)}$; we may assume $0\leq\theta_f(p)\leq\pi$.
Let $I$ be any subset of $[0,\pi]$, and let 
\[
\pi_I(x)=\#\{p:\text{prime}\;|\;p\leq x, \;\theta_f(p)\in I\}.
\]
Then Thorner's result is, under Assumption~\ref{automorphy},
\begin{align}\label{Thorner}
\frac{\pi_I(x)}{\pi(x)}=\frac{2}{\pi}\int_a^b \sin^2\theta d\theta
+O\left(\frac{x}{\pi(x)(\log x)^{9/8-\varepsilon}}\right)
\end{align}
for any $\varepsilon>0$, where $I=[a,b]$ and $\pi(x)$ denotes the number of primes up to $x$.
(Under the assumption of the GRH for $L(\mathrm{Sym}_f^{\gamma},s)$,
sharper estimates for
the error term are known.) 

\begin{theorem}\label{thm-general}
Let $N$ be a positive integer.
Let $f\in S_k(N)$ be a primitive form which is not of CM-type.
Let $\gamma\geq 2$, and assume 
Assumption~\ref{automorphy}.
Then, 
for any $\sigma>1-1/(\gamma+1)$,
there exists a continuous non-negative function
${\mathcal M}_{\sigma}(z,L(\mathrm{Sym}_f^{\gamma},\cdot))$, explicitly defined on $\mathbb{C}$, for 
which
\begin{align}\label{6-9}
&\lim_{T\to\infty}\frac{1}{2T}\int_{-T}^T \Phi(\log L(\mathrm{Sym}_f^{\gamma},s+i\tau))d\tau\\
=&
\int_{\mathbb{C}} {\mathcal M}_{\sigma}(z,L(\mathrm{Sym}_f^{\gamma},\cdot))\Phi(z)|dz|\notag
\end{align}
holds for any test function $\Phi$ as in the statement of Theorem \ref{thm-IM1114}.
\end{theorem}

\begin{remark}
In Theorem \ref{thm-general} we assume 
Assumtion~\ref{automorphy}, because it is not yet fully proved.
However, Barnet-Lamb et al. \cite{BGHT} proved the ``potential automorphy'' of
$L(\mathrm{Sym}_f^{\gamma},s)$, which gives a certain functional equation
(see \cite[Theorem B, Assertion 2]{BGHT}).    If the factors appearing in the
functional equation are shown to be sufficiently well-behaved, we can apply 
the result in
\cite{KST02} to obtain some suitable mean value result unconditionally, so we can
remove Assumtion~\ref{automorphy} from the statement of Theorem \ref{thm-general}.
\end{remark}

\section{Some general lemmas}\label{sec7}

We start the proof of theorems stated in the preceding section.
In this section we consider the general situation that
$\varphi\in\mathscr{M}_{00}$ with
$f(k,n)=1$ for all $k$ and $n$.    Then
\begin{equation}\label{7-1}
z_n(\theta_n;\varphi)=-\sum_{k=1}^{g(n)}
\log(1-a_n^{(k)}p_n^{-\sigma}e^{2\pi i\theta_n}).
\end{equation}
Put
\[
R_n(X;\varphi)=-\sum_{k=1}^{g(n)}\log(1-a_n^{(k)}X),
\]
and write its Taylor expansion as $R_n(X;\varphi)=\sum_{j=1}^{\infty}r_{j,n} X^j$.
Then we have 
\begin{equation}\label{7-2}
z_n(\theta_n;\varphi)=R_n(p_n^{-\sigma}e^{2\pi i\theta_n};\varphi)
=\sum_{j=1}^{\infty}r_{j,n} p_n^{-j\sigma}e^{2\pi ij\theta_n}.
\end{equation}

Let $x_n(\theta_n;\varphi)=\Re z_n(\theta_n;\varphi)$ and
$y_n(\theta_n;\varphi)=\Im z_n(\theta_n;\varphi)$.
Write $w=|w|e^{i\tau}=|w|\cos\tau+i|w|\sin\tau$.    Then
\begin{equation}\label{7-3}
\langle z_n(\theta_n;\varphi),w\rangle=|w|g_{\tau,n}(\theta_n;\varphi),
\end{equation}
where
\[
g_{\tau,n}(\theta_n;\varphi)=x_n(\theta_n;\varphi)\cos\tau
+y_n(\theta_n;\varphi)\sin\tau.
\]
Substituting this into \eqref{5-2}, we have
\begin{equation}\label{7-4}
K_n(w,\varphi)=\int_0^1 \exp\left(i|w|g_{\tau,n}(\theta_n;\varphi)\right)d\theta_n.
\end{equation}
Therefore, to evaluate $K_n(w,\varphi)$, the essential point is to analyze the
behavior of $g_{\tau,n}(\theta_n;\varphi)$.     We prove

\begin{lemma}\label{lem-7-1}
Let $\varphi\in\mathscr{M}_{00}$.
The function $g_{\tau,n}(\theta_n;\varphi)$ is a $C^{\infty}$-class function as a 
function in $\theta_n$.    Moreover, if $n$ is sufficiently large, and
\begin{equation}\label{7-5}
|r_{1,n}|\geq C
\end{equation}
holds with a positive constant $C$, then for those $n$,
$g_{\tau,n}^{\prime\prime}(\theta_n;\varphi)$ has exactly two zeros on the interval
$[0,1)$.    The same assertion also holds for 
$g_{\tau,n}^{\prime}(\theta_n;\varphi)$.
\end{lemma}

\begin{proof}
This lemma is an analogue of \cite[Lemma 7.1]{MU2}.
From the definition, we have
\begin{equation}\label{7-6}
r_{j,n}=\frac{1}{j}\sum_{k=1}^{g(n)}(a_{n}^{(k)})^j.
\end{equation}
Since $\varphi\in\mathscr{M}_{00}$, we find that 
$|r_{j,n}|\leq g(n)/j \leq C_0/j$.
Noting this point, we can see that exactly the same argument as in the proof of
\cite[Lemma 7.1]{MU2} can be applied to our present situation.
(The part on $g_{\tau,n}^{\prime}(\theta_n;\varphi)$ is the same as in
\cite[Remark 7.2]{MU2}.)
\end{proof}

Now we can show the following lemma, which is the analogue of Lemma \ref{lem2-MU2}
for $\varphi\in\mathscr{M}_{00}$.

\begin{lemma}[The Jessen-Wintner inequality for $\varphi$]\label{lem-JW}
Let $\varphi\in\mathscr{M}_{00}$,
and assume that $n$ is sufficiently large and
\eqref{7-5} holds.    Then we have
\begin{equation}\label{7-7}
K_n(w,\varphi)=O\left(\frac{p_n^{\sigma/2}}{|w|^{1/2}}+
\frac{p_n^{\sigma}}{|w|}\right).
\end{equation}
\end{lemma}

\begin{proof}
The method of the proof is the same as in \cite[Proposition 7.3]{MU2} (whose idea
goes back to Jessen and Wintner \cite{JW35}), so we just sketch the idea briefly.

Using \eqref{7-2} we have
$$
g_{\tau,n}(\theta_n;\varphi)=\sum_{j=1}^{\infty}|r_{j,n}|p_n^{-j\sigma}
\cos(\gamma_{j,n}+2\pi j\theta_n-\tau),
$$
where $\gamma_{j,n}=\arg r_{j,n}$, and hence
\begin{align*}
&g_{\tau,n}^{\prime}(\theta_n;\varphi)=-2\pi |r_{1,n}|p_n^{-\sigma}
\sin(\gamma_{1,n}+2\pi\theta_n-\tau)+O(p_n^{-2\sigma}),\\
&g_{\tau,n}^{\prime\prime}(\theta_n;\varphi)=-(2\pi)^2 |r_{1,n}|p_n^{-\sigma}
\cos(\gamma_{1,n}+2\pi\theta_n-\tau)+O(p_n^{-2\sigma}).
\end{align*}

Let $\theta_n=\theta_1^c, \theta_2^c$ be two solutions of 
$\cos(\gamma_{1,n}+2\pi\theta_n-\tau)=0$
($0\leq\theta_n<1$).
Then, when $n$ is sufficiently large and \eqref{7-5} holds, the two solutions of
$g_{\tau,n}^{\prime\prime}(\theta_n;\varphi)=0$ stated in Lemma \ref{lem-7-1} are close to
$\theta_1^c, \theta_2^c$.
Similarly, the two solutions of $g_{\tau,n}^{\prime}(\theta_n;\varphi)=0$ are close to
the two solutions $\theta_n=\theta_1^s, \theta_2^s$ be two solutions of 
$\sin(\gamma_{1,n}+2\pi\theta_n-\tau)=0$.
Then, for each $i,j$ ($1\leq i,j \leq 2$), there exists a unique $\theta_{ij}$ between
$\theta_i^c$ and $\theta_j^s$ for which
\[
|\sin(\gamma_{1,n}+2\pi\theta_n-\tau)|=|\cos(\gamma_{1,n}+2\pi\theta_n-\tau)|=1/\sqrt{2}
\]
holds.     

We divide the interval $0\leq\theta_n<1$ (mod 1) into four subintervals at the values
$\theta_{ij}$, and divide also the integral \eqref{7-4} accordingly.
   
On two of those subintervals 
$|\sin(\gamma_{1,n}+2\pi\theta_n-\tau)|\geq 1/\sqrt{2}$, which implies that
$|g_{\tau,n}^{\prime}(\theta_n;\varphi)|$ is not close to 0.
Therefore the integrals on those subintervals can be evaluated by the first
derivative test.    On the other two subintervals 
$|g_{\tau,n}^{\prime\prime}(\theta_n;\varphi)|$ is not close to 0, so the second
derivative test works.
These evaluations give the conclusion \eqref{7-7}.
\end{proof}

If there exist at least five large values of $n$ for which \eqref{7-5} holds, then
we can apply Lemma \ref{lem-JW} to Lemma \ref{lem-MU2} to obtain 
\begin{equation}\label{7-8}
W_{\sigma}(R;\varphi)=\int_R {\mathcal M}_{\sigma}(z,\varphi)|dz|
\end{equation}
for any $\sigma>\sigma_0$, with an explicitly constructed continuous
non-negative function ${\mathcal M}_{\sigma}(z,\varphi)$ (the associated
$M$-function).    
Then, as indicated in Remark \ref{rem-MU2}, we can deduce the formula of the form
\begin{equation}\label{7-9}
\lim_{T\to\infty}\frac{1}{2T}\int_{-T}^T \Phi(\log\varphi(s+i\tau))d\tau=
\int_{\mathbb{C}} {\mathcal M}_{\sigma}(z,\varphi)\Phi(z)|dz|
\end{equation}
in the region $\sigma>\sigma_0$, for any test function $\Phi$ as in the statement of
Theorem \ref{thm-IM1114}.
Therefore, to complete the proof of our theorems, the only remaining task is to
show \eqref{7-5} for sufficiently many large values of $n$.

\section{Proof of Theorems \ref{thm-square} and \ref{thm-general}}\label{sec8}

Now we return to the specific situation of symmetric power $L$-functions.

\begin{proof}[Proof of Theorem \ref{thm-square}]
In this case, for any $n$ such that $p_n\nmid N$, we see that $g(n)=3$, and
from \eqref{7-6} we have
\begin{align}\label{8-1}
r_{1,n}&=\alpha_f^2(p_n)+\alpha_f(p_n)\beta_f(p_n)+\beta_f^2(p_n)\\
&=(\alpha_f(p_n)+\beta_f(p_n))^2 -\alpha_f(p_n)\beta_f(p_n)\notag\\
&= (\lambda_f(p_n))^2-1. \notag
\end{align}
If $p_n\in\mathbb{P}_f(\varepsilon)$, then $|\lambda_f(\varepsilon)|>\sqrt{2}-\varepsilon$, 
so
\[
r_{1,n}>(\sqrt{2}-\varepsilon)^2-1=1-(2\sqrt{2}\varepsilon-\varepsilon^2),
\]
which is positive if $\varepsilon$ is small.    Since $\mathbb{P}_f(\varepsilon)$ is a
set of positive density, we now obtain the inequality \eqref{7-5} for infinitely
many values of $n$.    
This completes the proof.
\end{proof}

\begin{proof}[Proof of Theorem \ref{thm-general}]
In this case, for any $n$ such that $p_n\nmid N$, 
\[
r_{1,n}=\sum_{h=0}^{\gamma}\alpha_f^{\gamma-h}(p_n)\beta_f^h(p_n).
\]
In particular $r_{1,n}$ is real, and so
\[
r_{1,n}=\Re r_{1,n}=\sum_{h=0}^{\gamma}\cos((\gamma-2h)\theta_f(p_n)).
\]
Then it is easy to see that
\[
r_{1,n}\sin\theta_f(p_n)=\sin((\gamma+1)\theta_f(p_n))
\]
(cf. \cite[p.86]{RMKM}), hence
\begin{equation}\label{8-2}
|r_{1,n}|\geq |\sin((\gamma+1)\theta_f(p_n))|.
\end{equation}

Fix a number $\xi\in(0,\pi/2)$, and let $\eta=\sin\xi$.    Then $0<\eta<1$. 
Define the intervals
\[
A(j)=\left[\frac{2\pi j+\xi}{\gamma+1},\frac{2\pi j+\pi-\xi}{\gamma+1}\right],\;
B(j)=\left[\frac{2\pi j+\pi+\xi}{\gamma+1},\frac{2\pi j+2\pi-\xi}{\gamma+1}\right]
\]
where $j$ is a non-negative integer.    If $\gamma$ is odd, then
\begin{equation}\label{8-3}
|\sin((\gamma+1)\theta_f(p_n))|\geq \eta
\end{equation}
if and only if
\begin{equation}\label{8-4}
\theta_f(p_n)\in I_1:=\bigcup_{j=0}^{(\gamma-1)/2}\left(A(j)\cup B(j)\right).
\end{equation} 
If $\gamma$ is even, then \eqref{8-3} holds if and only if
\begin{equation}\label{8-5}
\theta_f(p_n)\in I_2:=\bigcup_{j=0}^{(\gamma-2)/2}\left(A(j)\cup B(j)\right)
\cup A(\gamma/2).
\end{equation} 
These observations and \eqref{8-2} and \eqref{8-3} imply that $|r_{1,n}|\geq\eta$ if and
only if $\theta_f(p_n)\in I_1$ (if $\gamma$ is odd) or $\in I_2$ (if $\gamma$ is even).
Therefore, to prove Theorem \ref{thm-general}, it is enough to show that the set
\begin{equation}\label{8-6}
\{p:\text{prime}\;|\;\theta_f(p_n)\in I_\ell \}\quad(\ell=1,2)
\end{equation}
is of positive density.

Since
\[
\int_a^b \sin^2\theta d\theta=
\frac{1}{2}\left(b-a-\frac{1}{2}(\sin 2b-\sin 2a)\right),
\]
from \eqref{Thorner} we have
\begin{align}\label{8-7}
\frac{\pi_I(x)}{\pi(x)}=\frac{1}{\pi}\left(b-a-\frac{1}{2}(\sin 2b-\sin 2a)\right)
+O\left((\log x)^{-1/8+\varepsilon}\right)
\end{align}
for $I=[a,b]$.
Denote
\begin{align*}
&a_{A(j)}=\frac{2\pi j+\xi}{\gamma+1}, \;\;b_{A(j)}=\frac{2\pi j+\pi-\xi}{\gamma+1},\\
&a_{B(j)}=\frac{2\pi j+\pi+\xi}{\gamma+1},\;\;
b_{B(j)}=\frac{2\pi j+2\pi-\xi}{\gamma+1}.
\end{align*}
Then from \eqref{8-7} we can write
\begin{equation}\label{8-8}
\frac{\pi_{I_\ell}(x)}{\pi(x)}=\frac{1}{\pi}S_{\ell}+\frac{1}{2\pi}T_{\ell}
+O\left((\log x)^{-1/8+\varepsilon}\right)\qquad (\ell=1,2),
\end{equation}
where
\begin{align*}
&S_1=\sum_{j=0}^{(\gamma-1)/2}\left(b_{A(j)}-a_{A(j)})+(b_{B(j)}-a_{B(j)})\right),\\
&S_2=\sum_{j=0}^{(\gamma-2)/2}\left(b_{A(j)}-a_{A(j)})+(b_{B(j)}-a_{B(j)})\right)
+\left(b_{A(\gamma/2)}-a_{A(\gamma/2)}\right),\\
&T_1=\sum_{j=0}^{(\gamma-1)/2}\left((\sin(2b_{A(j)})-\sin(2a_{A(j)}))
+(\sin(2b_{B(j)})-\sin(2a_{B(j)}))\right),\\
&T_2=\sum_{j=0}^{(\gamma-2)/2}\left((\sin(2b_{A(j)})-\sin(2a_{A(j)}))
+(\sin(2b_{B(j)})-\sin(2a_{B(j)}))\right)\\
&\qquad\qquad+\left(\sin(2b_{A(\gamma/2)})-\sin(2a_{A(\gamma/2)})\right).
\end{align*}
It is easy to see that
\begin{equation}\label{8-9}
S_{\ell}=\pi-2\xi \qquad (\ell=1,2).
\end{equation}
Next we show that
\begin{equation}\label{8-10}
T_{\ell}=0 \qquad (\ell=1,2).
\end{equation}

In fact, we know
\[
(\sin(2b_{\Box(j)})-\sin(2a_{\Box(j)}))
=2\sin\frac{\pi-2\xi}{\gamma+1}\cos\frac{4\pi j +c\pi}{\gamma +1},
\]
where $c=1$ if $\Box=A$ and $c=3$ if $\Box=B$.
Then
\begin{align*}
T_1&=2\sin\frac{\pi-2\xi}{\gamma+1}\sum_{j=0}^{(\gamma-1)/2}\left(
\cos\frac{4\pi j+\pi}{\gamma+1}+\cos\frac{4\pi j+3\pi}{\gamma+1}\right)\\
&=4\sin\frac{\pi-2\xi}{\gamma+1}\cos\frac{\pi}{\gamma+1}
\sum_{j=0}^{(\gamma-1)/2}\cos\frac{4\pi j+2\pi}{\gamma+1},
\end{align*}
and
\begin{align*}
\sin\frac{2\pi}{\gamma+1}\sum_{j=0}^{(\gamma-1)/2}\cos\frac{4\pi j+2\pi}{\gamma+1}
&=\frac{1}{2}\sum_{j=0}^{(\gamma-1)/2}\left(
\sin\frac{4\pi(j+1)}{\gamma+1}-\sin\frac{4\pi j}{\gamma+1}\right)\\
&=\frac{1}{2}(\sin(2\pi)-\sin 0)=0,
\end{align*}
therefore $T_1=0$.    Similarly we find that
\begin{align*}
T_2=4\sin\frac{\pi-2\xi}{\gamma+1}\cos\frac{\pi}{\gamma+1}
\sum_{j=0}^{(\gamma-2)/2}\cos\frac{4\pi j+2\pi}{\gamma+1}
+2\sin\frac{\pi-2\xi}{\gamma+1}\cos\frac{\pi}{\gamma+1},
\end{align*}
and the sum on the right-hand side is equal to $-1/2$, and hence $T_2=0$.

From \eqref{8-8}, \eqref{8-9} and \eqref{8-10} we obtain
\begin{align}\label{8-11}
\frac{\pi_{I_\ell}(x)}{\pi(x)}=1-\frac{2\xi}{\pi}
+O\left((\log x)^{-1/8+\varepsilon}\right)\qquad (\ell=1,2).
\end{align}
Since $\xi<\pi/2$, this implies that the set \eqref{8-6} is of positive density in the
set of all primes.    This completes the proof.
\end{proof}

\begin{remark}
Actually, to prove Theorem \ref{thm-general}, it is not necessary to invoke the
quantitative result of Thorner \cite{Tho14}.    The above argument, combined with the famous 
solution of the Sato-Tate conjecture \cite{BGHT}, implies 
\begin{align}\label{8-12}
\frac{\pi_{I_\ell}(x)}{\pi(x)}\sim 1-\frac{2\xi}{\pi}>0,
\end{align}
which is sufficient for our purpose.
However we may expect that a quantitative formula like \eqref{8-11} will be useful when
we try to develop more detailed study on $M$-functions.
\end{remark}

%
%
\section*{Acknowledgement}
A part of the contents of this paper was presented at the Intrenational Conference on
Number Theory, celebrating the 130th birth anniversary of Srinivasa Ramanujan, 
held in IIT Ropar, Dec 2017.     The first author expresses his gratitude to
Professor Sanoli Gun and Professor Tapas Chatterjee for their kind invitation
and hospitality. 
Research
of the first author is
supported by Grant-in-Aid for Science Research (B) 18H01111, and that
of the
second author is by Grant-in-Aid for Young Scientists (B) 23740020, JSPS.

\end{document}